\documentclass[ECP,preprint]{ejpecp}

\usepackage[english]{babel}
\usepackage{amsmath,amssymb,amsthm, bbm, fixmath,comment}
\usepackage{enumerate}
\hypersetup{
  colorlinks=true,
  citecolor=blue
}
\usepackage{mathtools}
\usepackage[mathscr]{eucal}
\numberwithin{equation}{section}

\usepackage[
    natbib=true,
    backend=biber,
    style=alphabetic,
    sortlocale=auto,
    natbib=true,
    url=false, 
    doi=true,
    eprint=true,
    safeinputenc=true,
    giveninits=true,
    maxbibnames=10
]{biblatex}
\SHORTTITLE{High-dimensional limit of one-pass SGD on least squares}

\TITLE{High-dimensional limit of one-pass SGD on least squares.}

\AUTHORS{%
  Elizabeth Collins--Woodfin\footnote{McGill University, Canada.
    \EMAIL{elizabeth.collins-woodfin@mail.mcgill.ca}
    \url{https://sites.google.com/view/e-collins-woodfin}}
  \and 
  Elliot Paquette \footnote{McGill University, Canada. \EMAIL{elliot.paquette@mcgill.ca} \url{http://elliotpaquette.github.io}}}


\AMSSUBJ{60H30} 

\SUBMITTED{ TBD } 
\ACCEPTED{ TBD } 

\VOLUME{0}
\YEAR{0}
\PAPERNUM{0}
\DOI{0}

\ABSTRACT{
We give a description of the high-dimensional limit of one-pass
single-batch stochastic gradient descent (SGD) on a least squares problem.
This limit is taken with non-vanishing step-size, and with proportionally related
number of samples to problem-dimensionality.
The limit is described in terms of a stochastic differential equation
in high dimensions, which is shown to approximate the state evolution of SGD.
As a corollary, the statistical risk is shown to be approximated by 
the solution of a convolution-type Volterra equation with vanishing errors as dimensionality tends to infinity.
The sense of convergence is the weakest that shows that statistical risks of the two processes coincide.
This is distinguished from existing analyses by the type of high-dimensional limit given 
as well as generality of the covariance structure of the samples.
}

\begin{document}




\renewcommand{\a}{\alpha}
\renewcommand{\b}{\beta}
\renewcommand{\d}{\delta}
\newcommand{\w}{\omega}
\newcommand{\e}{\varepsilon}
\newcommand{\la}{\lambda}
\newcommand{\g}{\gamma}
\newcommand{\tx}{\tilde{x}}
\newcommand{\ba}{\mathbold{a}}
\newcommand{\bu}{\mathbold{u}}
\newcommand{\bb}{\mathbold{b}}
\newcommand{\bx}{\mathbold{x}}
\newcommand{\btx}{\tilde{\mathbold{x}}}
\newcommand{\bM}{\mathbold{M}}
\newcommand{\bK}{\mathbold{K}}
\newcommand{\bA}{\mathbold{A}}
\newcommand{\bB}{\mathbold{B}}
\newcommand{\bI}{\mathbold{I}}
\newcommand{\be}{\mathbold{e}}
\newcommand{\bv}{\mathbold{v}}
\newcommand{\bw}{\mathbold{w}}
\newcommand{\bV}{\mathbold{V}}
\newcommand{\bX}{\mathbold{X}}
\newcommand{\bm}{\mathbold{m}}

\newcommand{\PP}{\mathbb{P}}
\newcommand{\EE}{\mathbb{E}}
\newcommand{\RR}{\mathbb{R}}
\newcommand{\NN}{\mathbb{N}}
\newcommand{\CC}{\mathbb{C}}

\newcommand{\cN}{\mathcal{N}}
\newcommand{\cM}{\mathcal{M}}
\newcommand{\cF}{\mathcal{F}}
\newcommand{\cL}{\mathcal{L}}
\newcommand{\cE}{\mathcal{E}}

\newcommand{\cP}{\mathscr{P}}
\newcommand{\cR}{\mathscr{R}}

\newcommand{\tr}{\operatorname{tr}}
\newcommand{\Var}{\operatorname{Var}}
\newcommand{\grad}{\operatorname{lin}}
\newcommand{\hess}{\operatorname{quad}}

\newcommand{\Cr}[1]{\textbf{\textcolor{red}{{#1}}}}
\newcommand{\Cb}[1]{\textbf{\textcolor{blue}{{#1}}}}

\newcommand{\dif}{\mathrm{d}}

\section{Introduction}
Stochastic optimization methods are the modern day standard for many large-scale computational tasks,
especially those that arise in machine learning.
There is a long history of analyses of these algorithms, beginning with the seminal work of \cite{robbins1951},
which focused on long-time behavior in a fixed dimensional space.
However, modern applications of stochastic optimization have motivated a different regime of analysis,
where the problem dimensionality grows proportionally with the run-time of the algorithm.

In this article, we derive the exact scaling behavior of stochastic gradient descent (SGD) on a least squares problem, in the \emph{one-pass} setting (see below) when dimension tends to infinity.  We further draw a comparison to the recent work \cite{PPAP01,Paquettes02}, in which the \emph{multi-pass} version of this problem was considered.

\paragraph{Stochastic gradient descent for empirical risk minimization}
Most versions of (minibatch) SGD can be formulated in the context of \emph{finite-sum problems}: 
\begin{equation} \label{eq:finite_sum}
    \min_{\bx \in \mathbb{R}^d}~ \bigg \{ f(\bx) \coloneqq \frac{1}{n} \sum_{i=1}^n f_i(\bx) \bigg\}.
\end{equation}
Empirical risk minimization fits in this context by supposing that 
there are $n$ independent samples from some data distribution, 
and each $f_i$ represents the loss of how the parameters $\bx$ in some model fit the $i$-th datapoint.
In this article we will exclusively consider the case of linear regression with $\ell^2$--regularizer.
So we suppose that there are $n$ iid samples $\left( (\ba_i, \bb_i) : 1 \leq i \leq n \right)$
from some distrbution $\mathcal{D}$,
with some assumptions to be specified.  We arrange this data into a design matrix $\bA$ and label vector $\bb$,
whose $i$-th row is given by $\ba_i$.
Finally, we specify the functions $f_i$ in \eqref{eq:finite_sum} by setting
\[
  f_i(\bx) = \tfrac12( \ba_i \cdot \bx - \bb_i)^2 + \tfrac{\delta}{2} \|\bx\|^2.  
\]
The parameter $\delta \geq 0$ is fixed and is the strength of the regularizer and throughout $\|\cdot\|$ will be the Euclidean norm.

Minibatch stochastic gradient descent in this context can be described as
\begin{equation} \begin{aligned} \label{eq:sgd}
    \bx_{k+1} &= \bx_{k} - \gamma_k \nabla f_{i_k}(\bx_k) \\
    &= \bx_k - \gamma_k \bA^T \be_{i_k} \be_{i_k}^T ( \bA \bx_k - \bb) 
    - \gamma_k \delta 
    \bx_k\,
\end{aligned} \end{equation}
where $\{\gamma_k\}$ are stepsize parameters, 
$\be_i$ is the $i$-th standard basis vector,  
and $\{i_k\}$ is a sequence of choices data.

In this article we consider the \emph{one-pass} case, in which $i_k = k$ but the algorithm is terminated after $n$ steps.
In practice, the order of the data points might be shuffled once before, but in the setting we have posed, with iid data, there is no point to including this additional randomization.
There are other choices for how to pick $i_k,$ and we highlight three of them, all of which are \emph{multi-pass} variants. 

In \emph{random (with replacement) sample} SGD, each $i_k$ is chosen uniformly at random from $\{1,2,\dots,n\}$.
This is the setting considered in \citep{PPAP01}, and we shall refer simply to this flavor of SGD simply as \emph{multi-pass} SGD in the bulk of the paper.
But for context, we also mention \emph{single shuffle} SGD, in which one takes $i_k = k \text{ mod } n$, and so only differs from one-pass SGD
in that the algorithm performs the same operations every \emph{epoch}\footnote{We take epoch to mean $n$ steps of the algorithm in all these contexts (including the with-replacement case).}
In \emph{random shuffle} SGD, one modifies the above strategy by randomly permuting the data between each epoch.
All of these strategies are extensively studied in the optimization literature: it is generally thought that the \emph{single shuffle} and \emph{random shuffle} strategies are faster than the random sample strategy \citep{YunSraJadbabaie} 
(see also 
\citep{RechtRe,
  GurbuzbalabanShuffling,
  SafranShamir,
  AhnYunSra}).

The one-pass case is the fundamental point of comparison for all of these methods,
being both simpler phenomenologically and also representing an idealization of SGD, 
in which the run-time of the algorithm is the amount data.
Appropriately, running for longer (meaning increasing $n$)
can only improve the statistical performance of the SGD estimator $\bx_n$,
in which context this is usually referred to as \emph{streaming SGD}.

The performance of SGD is measured through the \emph{population risk} $\mathscr{P}$
and sometimes through an $\ell_2$--regularized risk $\mathscr{R}$, which are given by
  \begin{align}
    &\mathscr{P}(\bx) \coloneqq  \tfrac12\EE_{(\ba,\bb)}( \ba \cdot \bx - \bb)^2,
    \quad
    &(\ba,\bb) \sim \mathcal{D},
    \label{eq:prisk} \\
    &\mathscr{R}(\bx) \coloneqq \tfrac12 \EE_{(\ba,\bb)} ( \ba \cdot \bx - \bb)^2 + \tfrac \delta 2 \|\bx\|^2,
    \quad
    &(\ba,\bb) \sim \mathcal{D}.
    \label{eq:rrisk}
  \end{align}
This regularized risk appears naturally as the mean behavior of one-pass SGD, in that
\[
  \bx_{k+1} = \bx_{k} - \gamma_k (\nabla \mathscr{R}(\bx_k) + \xi_{k+1}),
\]
for martingale increments $(\xi_k : 1 \leq k \leq n)$.  
Another natural statistical setting to consider is out-of-distributional regression, 
in which case we would measure the performance of SGD trained on $\mathcal{D}$ but tested, as in \eqref{eq:prisk} after replacing $\mathcal{D}$ with another distribution $\mathcal{D}'$.  
We shall not pursue this case in detail, 
but we note that all of the above examples are some quadratic functionals of the SGD state $\bx$.

\paragraph{Data and stepsize assumptions} 
The goal of this analysis is to allow the number of samples $n$ to be large
and proportional to the dimension of the problem, here $d$.
This means that the data must be normalized to be nearly dimension independent.
Further, we shall need good tail properties of some of the random variables involved,
and so we recall the Orlicz norms $\|\cdot\|_{\psi_p}$ for $p \geq 1$ which are given by
\[
  \|X\|_{\psi_p} = \inf\{ t : \EE e^{|X|^p/t^p} \leq 2\}.
\]
We refer the reader to \citep{Vershynin} for further exposition, properties and equivalent formulations.

We shall suppose throughout that under $\mathcal{D}$, the labels are given by an underlying linear model with noise.
Formally, we suppose that:
\begin{assumption}\label{a:lineargroundtruth}
  For $(\ba,\bb)$ sampled from $\mathcal{D}$, conditionally on $\ba$, the distribution of $\bb$ is given
  by $\ba \cdot \tilde{\bx} + w$ where $w$ is mean $0$, variance $\eta^2 \geq 0$ and is subgaussian with $\|w\|_{\psi_2} \leq d^{\e}.$
  The ground truth $\tilde{\bx}$ is assumed to have norm at most $d^{\e}$.
\end{assumption}
\noindent The constant $\e$ will be small and fixed throughout. Anything less than $\tfrac{1}{18}$ will do.

The data covariance is assumed to be normalized in such a way that it is almost bounded in norm, which is to say:
\begin{assumption}\label{a:covariance}
  The covariance matrix $\bK \coloneqq \EE[ \ba \ba^T ]$ has operator norm bounded independent of $d.$
\end{assumption}
\noindent Note that while we do not explicitly assume that $\ba$ is centered, the mean would have to be small in some sense to achieve the assumption above. 

Finally, we suppose that $\ba$ has good tail properties, namely that
\begin{assumption}\label{a:tails}
  The data vector $\ba$ satisfies that, for any deterministic $\bx$ of norm less than $1$, $\|\ba \cdot \bx\|_{\psi_2} \leq d^{\e}$, 
  and the data vector $\ba$ satisfies the Hanson-Wright inequality: for all $t \geq 0$ and for any deterministic matrix $\bB$
  \[
    \PP\left(\left|\ba^T\bB\ba-\EE\ba^T\bB\ba\right|\geq t\right)
    \leq
    2\exp\left(-\min\left\{\frac{t^2}{d^{4\e}||\bB||_{HS}^2},\;\frac{t}{d^{2\e}||\bB||}\right\}\right).
  \]
\end{assumption}

We remark that these assumptions hold for two important settings:
\begin{enumerate}[(a)]
\item when $\ba=\sqrt{\bK}\bu$ where $\bK$ is some deterministic matrix of bounded operator norm and $\bu$ is a vector of iid subgaussian random variables or
 \item when $\ba$ is a vector with the \emph{convex concentration property}, see \citep{adamczak} for details. 
\end{enumerate}
There are natural examples of the second case, such as random features models \citep{Rahimi2008Random} with Lipschitz activation functions (see also \citep[Proposition 6.2]{PPAP01} for specifics in the case of random features).
We note that by truncation, it is also possible to work in the setting (a) above but solely under a uniform bound on a sufficiently high but finite moment, but we do not pursue this.

Finally, the step-size parameters $\gamma_k$ must be normalized approporiately:
\begin{assumption}\label{a:steps}
  The stepsize $\gamma_k = \tfrac{\gamma}{d}$ for all $k$ and fixed $\gamma > 0$.
\end{assumption}
\noindent We note that we may also pick $\gamma_k = \gamma(k/d)/d$ for a bounded continuous function $\gamma: [0,\infty) \to [0,\infty)$, and this leads to no change anywhere in the arguments.

In light of all these assumptions, we note that SGD finally reduces to the following 
stochastic recurrence
\begin{equation}\label{eq:renormalized}
  \bx_k-\btx=
  (\bI(1-\tfrac{\gamma \d}{d})-\g\bm_k\bm_k^T)(\bx_{k-1}-\btx)
  -\tfrac{\gamma \d}{d} \btx
  +\g\bm_k\eta_k,
\end{equation}
where $\eta_k = w_k/\sqrt{d}$ and $\bm_k = \ba_k/\sqrt{d}$.

\paragraph{Homogenized SGD} 
Our theorem is most easily formulated as showing that the state of SGD can be compared to a certain diffusion model in high dimensions.
Homogenized SGD is defined to be a continuous time process with initial condition $\bX_0=\bx_0$ that solves the stochastic differential equation
\begin{equation}
\dif \bX_t=-\g\nabla\mathscr{R}(\bX_t) \dif t+\g\sqrt{\tfrac2d\mathscr{P}(\bX_t) \bK} \dif B_t
\end{equation}
where $B_t$ is standard Brownian motion in dimension $d$,
and $\mathscr{R}$ and $\mathscr{P}$ are the regularized and unregularized risks, respectively
(recall \eqref{eq:prisk} and \eqref{eq:rrisk}).

Our main theorem shows that for quadratic statistics (in particular the risks \eqref{eq:prisk} and \eqref{eq:rrisk}), homogenized SGD and SGD are interchangeable to leading order.
We use the probabilistic modifier \emph{with overwhelming probability} to mean a statement holds except  on an event of probability at most $e^{-\omega(\log d)}$ where $\omega(\log d)$ tends to $\infty$ faster than $\log d$ as $d\to\infty.$
We further introduce a norm $\|\cdot\|_{C^2}$ on quadratic functions $q : \RR^d \to \CC$
\[
\|q\|_{C^2} \coloneqq 
\|\nabla^2 q\|
+
\|\nabla q(0)\|
+
|q(0)|,
\]
with the norms on the right hand side being given by the operator and Euclidean norm respectively.
\begin{theorem}\label{thm:main}
For any quadratic $q :\RR^d \to \RR$,
and for any deterministic initialization $\bx_0$ with $\|\bx_0\| \leq 1$,
there is a constant $C\left( \|\bK\| \right)$ 
so that
the processes $\{\bx_k\}_{k=0}^n$ and $\{\bX_t\}_{t=0}^{n/d}$ 
satisfy
for any $n$ satisfying $ n \leq d \log d/C(\|\bK\|)$
\begin{equation}
\sup_{0\leq k\leq n}
\Big|q(\bx_k)-q(\bX_{k/d})\Big|<\|q\|_{C^2}\cdot 
e^{C\left( \|\bK\| \right)\frac{n}{d}}
\cdot
d^{-\frac12+9\e}
\end{equation}
with overwhelming probability.
\end{theorem}
\noindent The processes $\bx_k$ and $\bX_t$ are independent, and hence this is also a statement about concentration.  In particular, the statement is also true if we replace $q(\bX_{k/d})$ by $\EE q(\bX_{k/d})$. 

\paragraph{Explicit risk curves}
Using existing theory (see \cite[Theorem 1.1]{PPAP01}),
$\mathscr{P}(\bX_{k/d})$ 
and 
$\mathscr{R}(\bX_{k/d})$ 
can be seen to concentrate around their means, 
which solve a convolution Volterra equation.  Specifically,
$\EE\mathscr{P}(\bX_t)= \Psi(t)$ and 
$\EE\mathscr{R}(\bX_t)= \Omega(t)$
%
%
where
\[
  \begin{pmatrix}
\Psi(t) \\
\Omega(t) 
\end{pmatrix}
  = 
  \begin{pmatrix}
  \mathscr{P}( \mathscr{X}_{\gamma t})\\ 
  \mathscr{R}( \mathscr{X}_{\gamma t})\\ 
  \end{pmatrix}
  + \int_0^t 
\begin{pmatrix}
  K(t-s ; \nabla^2 \mathscr{R})\Psi(s) \\
K(t-s ; \nabla^2 \mathscr{P})\Psi(s) 
\end{pmatrix} \dif s,
  \quad
  \begin{cases}
    K(t ; \bM) \coloneqq \tfrac{\gamma^2}{d}
    \tr( \bK \bM e^{-2\gamma t(\bK + \delta \bI)}), \\
   \dif\mathscr{X}_{\gamma t} \coloneqq - \gamma \nabla \mathscr{R}(\mathscr{X}_{\gamma t}),
   \quad \mathscr{X}_{0} = \bX_0.
 \end{cases}
\]
The process $\mathscr{X}_t$ is gradient flow, and for the problems here, is explicitly solvable.
Note the equation for $\Psi$ is autonomous, and the solution of $\Omega$ is then solvable in terms of it.
From this Volterra equation, it is easy to derive convergence rates, as well as convergence thresholds, as well as expressions for the limiting risk (in the double scaling limit, $d\to\infty$ followed by $t \to \infty$).  See the discussions in \citep{Paquettes02} and see Figure \ref{fig:risks} for an example.

\paragraph{Comparison to the multi-pass case}

In \cite{PPAP01}, the analog of Theorem \ref{thm:main} was proven for the (random sample) multi-pass case.  In that case, the diffusion is different.  Introduce the \emph{empirical risk} $\mathscr{L}(\bx) = \tfrac{1}{2n}\| \bA \bx - \bb\|^2$ and the regularized empirical risk $f(\bx) = \mathscr{L}(\bx) + \delta \| \bx \|^2$.  Then the homogenized SGD for the multi-pass case becomes
\begin{equation}
\dif \bX_t=-\g\nabla f(\bX_t) \dif t+\g\sqrt{\tfrac2d\mathscr{L}(\bX_t) (\tfrac1n \bA^T \bA) } \dif B_t
\end{equation}
Hence the difference between the multi-pass and one-pass cases is that the population risks are traded for the empirical risk and the data covariance matrix $\bK$ is traded for the empirical data covariance matrix $(\tfrac1n \bA^T \bA)$.  Note that if one conditions on the data $(\bA,\bb)$, then multi-pass SGD in fact \emph{is} streaming SGD, but with a finitely supported data distribution (specifically the empirical distribution of data); however the empirical distribution of data is very far from satisfying Assumption \ref{a:tails}.

\begin{figure}[h]
  \begin{center}
    \includegraphics[width=10cm]{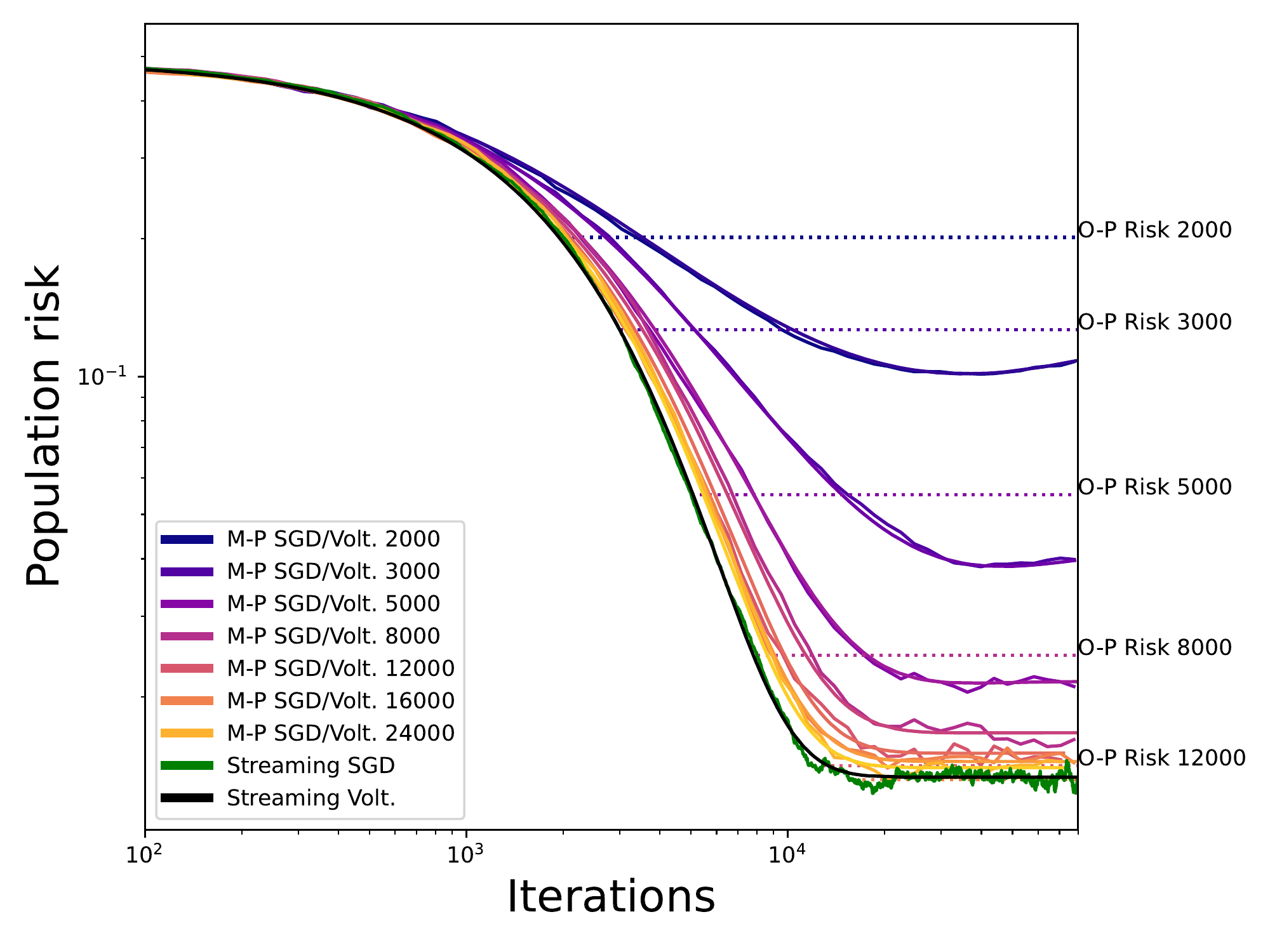}
  \end{center}
  \caption{Risk curves for a simple linear regression problem in $d=2000$. Multi-pass SGD, its high dimensional equivalent (the expected risk under homogenized SGD, i.e. ``Volterra''), Streaming SGD (i.e. one-pass with varying dataset size), and the expected risk of homogenized SGD (``Streaming Volterra'') are all plotted.  Risk levels for streaming SGD at various levels $n$ are plotted for comparison against the corresponding multi-pass version.  Note that at smaller dataset sizes, multi-pass SGD improves greatly over one-pass SGD.  At higher dataset sizes, they are similar and in fact multi-pass SGD always underperforms. }
  \label{fig:risks}
\end{figure}


\paragraph{Discussion}

We have presented an approach to taking the high-dimensional limit of one-pass SGD on a least squares problem in which the number of steps is proportional to the dimension of the problem.  The limit object is described in terms of a Langevin type diffusion, which can be directly compared to the same object in the multi-pass case.

The literature on scaling limits of one-pass SGD training is large, and so we mention just some of the closest literature.  \citep{gerard2} is perhaps the closest high-dimensional diffusion approximation, and it applies in cases where there is a hidden finite dimensional structure; it covers the case studied here when $\bK=\bI$ as well as cases in which $\bK$ has boundedly many eigenvalues.  See also \citep{gerard1}.

There are other scaling limits that pursue a different formulation than the one here.
\citep{Wang_2019,Wang_2019a} give a PDE limit for the state for a generalized linear model, with identity covariance.
\citep{BGH} give a scaling limit of the SGD under smoothness assumptions on the covariance $\bK,$ when interpreted geometrically; they further describe fluctuations of SGD in a certain sense.  Note that Theorem \ref{thm:main} essentially gives the law of large numbers for the risks and not the fluctuations.  

\citep{Gerbelot} uses dynamical field theory to give closely related results for shallow neural networks with minibatch SGD of large batch-size; dynamical mean field theory provides an implicit characterization of the autocorrelation of the minibatch noise and a few other processes.  See also the related work of \citep{Celentano}.  We comment that in the case of proportional batch sizes, there is also a discrete Volterra description in \cite{LeeChengPaquettePaquette}.

\paragraph{Organization} 
In Section \ref{sec:proof} we give an overview of the main proof, reducing it to its main technicalities. In Section \ref{sec:bounding}, we bound the stochastic error terms, representing the main technical contribution of the paper.

\section{Main argument of proof} \label{sec:proof}
In order to compare the SGD and homogenized SGD, 
we use a version of the martingale method in diffusion approximations (see \cite{Kurtz}).
In effect we show that $q(\bx_k)$ nearly satisfies the conclusion of It\^o's lemma.
Further, we show the martingale terms in both of the Doob decompositions 
are small, and hence it suffices to show the predictable parts 
of $q(\bx_k)$ and $q(\bX_t)$ are close. 

To advance the discussion, we compute this Doob decomposition.
To take advantage of the simpler structure afforded by removing $\btx$, introduce
\begin{equation}
\bv_k\coloneqq\bx_k-\btx
\quad
\text{and}
\quad
\bV_t\coloneqq \bX_t - \btx.
\end{equation}
We shall extend the first integer indexed function to real-valued indices by setting $\bv_t = \bv_{\lfloor t \rfloor}$.
We also let $(\cF_t : t \geq 0)$ be the filtration generated by $(\bv_t : t \geq 0)$ and $(\bV_{t/d} : t \geq 0)$.  Hence for all $k \in \NN$, $\bv_k$ is measurable with respect to $\cF_k$.
%
%
%
Recalling the recurrence \eqref{eq:renormalized}
for a quadratic $q$
\begin{equation}\label{eq:q_increment}\begin{split}
q(\bv_k)-q(\bv_{k-1})
&=-
\g(\nabla q(\bv_{k-1}))^T
(\bu_{k-1}+\Delta_k)
+\tfrac{\g^2}{2}
(\bu_{k-1}+\Delta_k)^T
(\nabla^2q)
(\bu_{k-1}+\Delta_k),
\quad
\text{where}
\quad 
\\
\Delta_k
&=
\bm_k(\bm_k^T\bv_{k-1} -\eta_k)
\quad\text{and}\quad
\bu_{k-1}
=\tfrac{\d}{d}(\bv_{k-1}+\btx).
\end{split}
\end{equation}
The equation above can each be decomposed as a predictable part and two martingale increments 
\begin{equation}
\label{eq:q_increment_decomp}\begin{split}
q(\bv_k)-q(\bv_{k-1})=
&-\g (\nabla q(\bv_{k-1}))^T
\bigl( (\tfrac\d d\bI + \tfrac1d \bK )\bv_{k-1} + \tfrac{\d \btx}{d}\bigr)
+\Delta\cM_{k}^{\grad}\\
&+\tfrac{\g^2}{2}\tr(\tfrac1d \bK (\nabla^2q))
\left(\tfrac1d\bv_{k-1}^T\bK \bv_{k-1}+\EE[\eta_k^2]\right)+\Delta\cE_k^{\hess}+\Delta\cM_{k}^{\hess},  \\
\text{where}
\quad\Delta\cM_{k}^{\hess}
\coloneqq
&\Delta_k^T \nabla^2 q \Delta_k- \EE[ \Delta_k^T \nabla^2 q \Delta_k \mid \cF_{k-1}].
\end{split}
\end{equation}
The remainder of the martingale increments are given by $\Delta\cM_{k}^{\grad}$ and are all linear in $\Delta_k$.
The predictable parts have been further decomposed into the leading order terms and an error term $\Delta\cE_k^{\hess}$.  

These predictable parts, in turn, depend on different statistics $q_1(\bv_{k-1})$.
In finite dimensional settings, we would be able to relate this (or some suitably large finite set of summary statistics $q,q_1, \dots, q_r$ to itself through a closed system of recurrences.  In this setting, this is not possible.  On the other hand, for the problem at hand, we show there is a manifold of functions which approximately closes.  Specifically, we let
\begin{equation}\label{eq:Qn}\begin{split}
  Q_n(q)
  \coloneqq
  Q_n(q,\bK)
  =&\Big\{
    q(\bx),
    \quad (\nabla q(\bx))^TR(z;\bK)\bx,
    \quad\bx^TR(y;\bK)(\nabla^2q)R(z;\bK)\bx, \\
    &
    \qquad (\nabla q(\bx))^TR(z;\bK)\btx,
    \quad\bx^TR(y;\bK)(\nabla^2q)R(z;\bK)\btx, \quad\forall\, z,y \in \Gamma
  \Big\}.
\end{split}
\end{equation}
Here $R(z;\bK) = (\bK - z\bI)^{-1}$ is the resolvent matrix, and $\Gamma$ is a circle of radius $\max\{1,3\|\bK\|\}$.
In order to control the martingales, it is convenient to impose a stopping time
\begin{equation}
\tau \coloneqq \inf\left\{k:||\bv_k||>d^\e\}\cup\{td:||\bV_t||>d^\e\right\},
\end{equation}
and we introduce the corresponding stopped processes
\begin{equation}
  \label{eq:stopped}
\bv_{k}^\tau=\bv_{k\wedge \tau},\quad \bV_t^\tau=\bV_{t\wedge(\tau/d)}.
\end{equation}
We prove a version of our theorem for the stopped processes and then show that the stopping time is greater than $n$ with overwhelming probability.

Our key tool for comparing $\bv_{td}$ and $\bV_t$ is the following lemma.

\begin{lemma}\label{lem:sgd_hsgd_comparison}
  Given a quadratic $q$ with $\|q\|_{C^2} \leq 1$, with $Q=Q_n(q) \cup Q_n(\cP) \cup Q_n(\|\cdot\|^2)$ 
  as above,
\begin{equation}\begin{split}
  \max_{0 \leq t \leq \tfrac{n}{d}}
  |q(\bv_{td}^\tau)-q(\bV_t^\tau)|
  \leq 
  &\sup_{0\leq t\leq {n}/{d}}
  \left(|\cM_{\lfloor td \rfloor}^{\grad,\tau}|+
  |\cM_{\lfloor td \rfloor}^{\hess,\tau}|
  +|\cE_{\lfloor td \rfloor}^{\hess,\tau}|
  +|\cM_t^{HSGD,\tau}|\right)\\
  &+C(\|\bK\|)
  \cdot
  \sup_{g\in Q}\int_0^{n/d}|g(\bv_{sd}^\tau)-g(\bV_s^\tau)|ds.
\end{split}\end{equation}
Here $\cM_t^{HSGD,\tau}$ is the martingale part in the semimartingale decomposition of $q(\bV_t^\tau)$.
\end{lemma}
\begin{proof}[Sketch of Proof]
  Owing to the similarities of this claim with the proof in \cite[Proposition 4.1]{PPAP01}, we just illustrate the main idea.
  The idea is that if we take a $g \in Q$, and we apply \eqref{eq:q_increment_decomp},
then in the predictable part of $g(\bv_t)$ we have
\[
  I_1 \coloneqq 
  \int_0^t 
  \nabla g(\bv_{sd})^T (\delta \bI + \bK) \bv_{sd}\dif s,
  \quad
  I_2 \coloneqq 
  \int_0^t 
  \nabla g(\bv_{sd})^T \btx 
  \dif s,
  \quad
  I_3 \coloneqq 
  \int_0^t 
  \bv_{sd}^T \bK \bv_{sd}
  \dif s.
\]
These also appear with coefficients that can be bounded solely using $\|g\|_{C^2}$ and $\|\bK\|$.
We get the same, applying It\^o's lemma to $g(\bV_t)$, albeit with the replacement $\bv_t \to \bV_t$.
We wish to bound for example $I_1(\bv_t)-I_1(\bV_t)$.
We do this by expressing its integrand as $p(\bv_t) - p(\bV_t)$ for polynomial $p$.
If $g$ is linear (the final row of \eqref{eq:Qn}), then $p$ is again linear.  For example, if it is $g(\bx) = \nabla q(\bx)^T R(z;\bK) \btx,$ then $p$ is again linear and is given by
\[
  p(\bx) =  
  \bx^T (\delta \bI + \bK)R(z;\bK) \btx
  =
  \delta \bx^T R(z;\bK) \btx
  +\bx^T R(z;\bK) \btx
  -z\bx^T \btx,
\]
where we have used the resolvent identity $(\bK -z)R(z;\bK)=\bI$.
Note the function $\bx^T R(z;\bK) \btx$ is contained in $Q$ by virtue of being in $Q_n(\|\cdot\|^2)$.  Moreover, by Cauchy's integral formula, we can represent $\bx^T \btx$ by averaging $\tfrac{-1}{2\pi i}\bx^T R(y;\bK) \btx$ over $y \in \Gamma$.
Hence
\[
  |p(\bv_{td})-p(\bV_t)| 
  \leq (\delta + 1 + 3\|\bK\|)\max_{g \in Q} |g(\bv_{td})-g(\bV_t)|. 
\]
The same manipulations lead finally to showing every term included in $Q$ can be controlled in a similar manner,
using the other elements of the class $Q$.  
\end{proof}

The second important idea is to discretize the set $Q$.
\begin{lemma}\label{lem:quad_net}
There exists $\bar{Q}\subseteq Q$ with $|\bar{Q}|\leq C(\|\bK\|) d^{4m}$ such that, for every $q\in Q$, there is some $\bar{q}\in\bar{Q}$ satisfying $||q-\bar{q}||_{C^2}\leq d^{-2m}$.
\end{lemma}
\begin{proof}
  On the spectral curve $\Gamma$, we can bound the norm of the resolvent.  Since
  \[
    \frac{\dif}{\dif z} R(z; \bK) = (\bK-z\bI)^{-2},
  \]
  we have it is norm bounded by an absolute constant.  The arc length of the curve is at most $C(\|\bK\|),$ and so by choosing a minimal net $d^{-2\e}$ of the manifold $\Gamma \times \Gamma$, the lemma follows. 
\end{proof}

Now the main technical part of the argument is to control the martingales and errors.  
As we work with the stopped process $\bv_k^\tau$ we introduce the stopped proccesses $\cM_k^{\grad,\tau},\cM_k^{\hess,\tau},\cE_k^{\hess,\tau}$, which are defined analogously to 
\eqref{eq:stopped}.
\begin{lemma}\label{lem:mart_bounds}
  For any quadratic $q$ with $\|q\|_{C^2}\leq 1$, 
the terms $\cM_k^{\grad,\tau},\cM_k^{\hess,\tau},\cE_k^{\hess,\tau}$ satisfy the following bounds with overwhelming probability (with a bound which is uniform in $q$)
for $n \leq d \log d$
\begin{enumerate}[(i)]
\item\label{item:Mgrad_bound} $\sup_{1\leq k\leq n}|\cM_{k}^{\grad,\tau}| \leq d^{-\frac12+5\e}$,
\item\label{item:Mquad_bound} $\sup_{1\leq k\leq n}|\cM_k^{\hess,\tau}|\leq d^{-\frac12+9\e}$,
\item\label{item:Equad_bound} $\sup_{1\leq k\leq n}|\cE_k^{\hess,\tau}|\leq d^{-1+9\e}.$
\end{enumerate}
\end{lemma}

Combining Lemmas \ref{lem:sgd_hsgd_comparison} and \ref{lem:quad_net}, along with the above (see also Lemma \ref{lem:mart_bound_HSGD} in which the homogenized SGD martingales are bounded), we conclude that, for any $\bar{q}\in\bar{Q}$ with $\|q\|_{C^2} = 1$, 
\begin{equation}
|\bar{q}(\bv_{td}^\tau)-\bar{q}(\bV_t^\tau)|
\leq  4 d^{-\frac12+9\e}
+C(\|\bK\|)\max_{g\in {Q}}\int_0^{t}|g(\bv_{sd}^\tau)-g(\bV_s^\tau)|ds.
\end{equation}
Hence by Lemma \ref{lem:quad_net} and by bounding $\|g\|_{C^2}$ over all $Q$,
\begin{equation}
\max_{g\in {Q}}|{q}(\bv_{td}^\tau)-{q}(\bV_t^\tau)|
\leq 
C(\|\bK\|)
\biggl(d^{-2}
+ d^{-\frac12+9\e}
+\int_0^{t}
\max_{g\in {Q}}|g(\bv_{sd}^\tau)-g(\bV_s^\tau)|ds
\biggr).
\end{equation}
By Gronwall's inequality, this gives us that with overwhelming probability
\begin{equation}
  \max_{g\in {Q}}
  \max_{0 \leq t \leq n/d}
  |{g}(\bv_{td}^\tau)-{g}(\bV_t^\tau)|\leq 
   C(\|\bK\|)
  (d^{-2} + 4 d^{-\frac12+9\e})e^{C(\|\bK\|)n/d}.
\end{equation}
Now we note that the norm function $\bx \mapsto \|\bx\|^2$ is one of the quadratics included in $Q.$  Hence if we let $\mathcal{G}$ be the event in the above display, and we let $\mathcal{E} = \{ \max_{0 \leq s \leq n/d}  \|\bV_s\| \leq d^{\e/2}\}$, then we have
\[
  \mathcal{G} 
  \cap
  \mathcal{E} \cap \{\tau \leq n/d\} 
  \subseteq 
  \{ 
  \|\bv_{\tau}\|-\|\bv_{\tau-1}\|
  \geq d^{\e/2}\}
  \cap
  \{\tau \leq n/d\}
  .
\]
This is because on the event $\{\tau \leq n/d\} \cap \mathcal{E}$ we must have had $\|\bv_\tau\| > d^{\e}$, but in the step before $\tau,$ we had $\bv_{\tau-1}$ could be compared to $\bV_{\tau-1}$ (due to $\mathcal{G}$, and we had the norm of $\bV_{\tau-1}$ was small.  Now it is easily seen that with overwhelming probability, no increment of SGD between time $0$ and $n/d$ can increase the norm by a power of $d$.  So to complete the proof it suffices to show $\mathcal{E}$ holds with overwhelming probability. 

Thus the proof is completed by the following:
\begin{lemma}
  For any $\delta > 0$ and any $t>0$
  with overwhelming probability
  \[
    \max_{0 \leq s \leq t} \|\bX_s\|^2
    \leq
    e^{C(\|\bK\|)t}
    d^{\d}.
  \]
  \label{lem:normbound}
\end{lemma}
\begin{proof}
  We apply It\^o's formula to $\phi(\bX_t)\coloneqq \log(1+\|\bX_t\|^2)$, from which we have
  \[
	\dif \phi(\bX_t)=
	-2\g \tfrac{\bX_t \cdot \nabla\mathscr{R}(\bX_t)}{1+\|\bX_t\|^2} \dif t+
	\tfrac{
	  \bX_t \cdot\g\sqrt{\tfrac2d\mathscr{P}(\bX_t) \bK}
	  \dif B_t
	}
	{1+\|\bX_t\|^2}
	+
	\bigl(\tfrac{\mathscr{P}(\bX_t)}
	{1+\|\bX_t\|^2}
	\tfrac{2\gamma^2}{d}\tr(\bK)
	-
	\tfrac{2\gamma^2\mathscr{P}(\bX_t) \bX_t^T \bK \bX_t}{d}\bigr) \dif t
  \]
  The drift terms and the quadratic variation terms can be bounded by some $C(\|\bK\|)$.
  Hence with this constant, for all $r \geq 0,$
  \[
    \PP( \max_{0 \leq s \leq t} \phi(\bX_s) \geq C(\|\bK\|)(t + r\sqrt{t})) \leq 2\exp(-r^2/2).
  \]
  Taking $r = \sqrt{\log d \log\log d}$, we conclude that with overwhelming probability
  \[
    \max_{0 \leq s \leq t} \phi(\bX_s)
    \leq C(\|\bK\|)(t + \sqrt{t \log d \log\log d}).
  \]
\end{proof}

\section{Controlling the errors}\label{sec:bounding}
The main goal of this section is to control the martingale terms and error terms; in particular we prove Lemma \ref{lem:mart_bounds}.  
In order to obtain these bounds, we will need the following concentration lemma, which is standard (c.f.\, \cite[Theorem 2.8.1]{Vershynin}, where the nonmartingale bound is proven.  The adaptation to the martingale case is a small extension):
\begin{lemma}[Martingale Bernstein inequality]
If $(M_n)_1^N$ is a martingale on the filtered probability space $(\Omega,(\cF_n)_1^N,\PP))$ and we define
\begin{equation}
  \sigma_{n,p}:=\left\|\inf\{t\geq0:\EE\left(e^{|M_n-M_{n-1}|^p/t^p}|\cF_{n-1}\right)\leq2\}\right\|_{L^\infty(\PP)},
\end{equation}
then there is an absolute constant $C>0$ so that, for all $t>0$,
\begin{equation}
\PP\left(\sup_{1\leq n\leq N}|M_n-\EE M_0|\geq t\right)\leq2\exp\left(-\min\left\{\frac{t}{C\max\sigma_{n,1}},\frac{t^2}{C\sum_1^N\sigma_{n,1}^2}\right\}\right).
\end{equation}
\end{lemma}
\noindent We will also record for future use an estimate on $\nabla q$ that follows from $\|\cdot\|_{C^2}$ control.
\begin{equation}\label{eq:q}
||\nabla q(\bx)||
\leq
||\nabla^2 q||\cdot ||\bx||+||\nabla q(0)|| 
\leq \|q\|_{C^2}\cdot( ||\bx|| + 1).
\end{equation}
\subsection{Martingale for gradient part of recurrence}
\begin{proof}[Proof of Lemma \ref{lem:mart_bounds} part (\ref{item:Mgrad_bound})]
Comparing \eqref{eq:q_increment} and \eqref{eq:q_increment_decomp}, we see that
for $k \leq \tau$
\begin{equation}\label{eq:SGD_Mgrad_increment}\begin{split}
\Delta\cM_k^{\grad,\tau}
&=\Big[
  \left(
  \bw_{k-1}
  ^T\bm_k\right)
\left(\bm_k^T\bv_{k-1}^\tau-\eta_k\right)
-\tfrac1d \bw_{k-1}^T
\bK\bv_{k-1}^\tau\Big] 
\eqqcolon
[\Delta\cM_k^{\grad1,\tau}-\Delta\cM_k^{\grad2,\tau}], \\
&\text{where}
\quad
\bw_{k-1} \coloneqq -\g \nabla q(\bv_{k-1}^\tau) + \tfrac{\g^2\d}{d}( \bv^\tau_{k-1}+\btx).
\end{split}\end{equation}
Note for $k > \tau$, the stopped martingale increment is $0.$
Using  \eqref{eq:q}, $\|\bw_{k-1}\| \leq C(\gamma,\d)d^{\e}.$
We will separately bound the contributions from $\Delta\cM_k^{\grad1,\tau}$ and $\Delta\cM_k^{\grad2,\tau}$ in terms of their Orlicz norms.  For the first part, for any fixed $k$, we condition on $\cF_{k-1}$ and Assumption \ref{a:tails}, we conclude
\begin{equation}\begin{split}
||\Delta\cM_{k}^{\grad1,\tau}||_{\psi_1}
\leq \left\| \bw_{k-1}^T\bm_k\right\|_{\psi_2}
\left\|\bm_k^T\bv_{k-1}^\tau-\eta_k\right\|_{\psi_2}
\leq Cd^{-\frac12+2\e}\cdot d^{-\frac12+2\e}
\end{split}\end{equation}
where $C$ is some absolute constant.
For the second part, we have
\begin{equation}
  |\Delta\cM_k^{\grad2,\tau}|
  =
  |\tfrac1d \bw_{k-1}^T\bK\bv_{k-1}^\tau|
\leq Cd^{-1+2\e}.
\end{equation}
Combining these, we see that, for every $k$,
\begin{equation}
\sigma_{k,1}:=\inf\{t>0:\EE[\exp(|
\Delta\cM_{k}^{\grad1,\tau}-\Delta\cM_{k}^{\grad2,\tau}
|/t)|\cF_{k-1}]\leq2\}\leq Cd^{-1+4\e}
\end{equation}
and, by the martingale Bernstein inequality,
\begin{equation}\begin{split}
\PP\left(\sup_{1\leq k\leq n} |\cM_{k}^{\grad,\tau}-\EE\cM_0^{\grad}|\geq t\right)
&\leq2\exp\left(-\min\left\{\frac{t}{c\max\sigma_{k,1}},\frac{t^2}{c\sum_{k=1}^n\sigma_{k,1}}
\right\}\right)\\
&\leq2\exp\left(-\min\left\{Ctd^{1-4\e},Ct^2d^{2-8\e}n^{-1}
\right\}\right).
\end{split}\end{equation}
As we assume that $n\leq d\log d$ then this gives us
\begin{equation}
\sup_{1\leq k\leq n}|\cM_{k}^{\grad,\tau}|
\leq d^{-\frac12+5\e}
\end{equation}
with overwhelming probability.
\end{proof}

\subsection{Martingale for Hessian part of recurrence}

\begin{proof}[Proof of Lemma \ref{lem:mart_bounds} parts (\ref{item:Mquad_bound}) and (\ref{item:Equad_bound})]
Next we consider the contribution from the Hessian part of the recurrence.  We write
\begin{equation}\label{eq:recurrence_quadpart}\begin{split}
&
\tfrac{\g^2}{2}(\bm_k\bm_k^T\bv_{k-1}^\tau-\bm_k\eta_k)^T(\nabla^2q)(\bm_k\bm_k^T\bv_{k-1}^\tau-\bm_k\eta_k)\\
&=\EE\left[\tfrac{\g^2}{2}(\bm_k\bm_k^T\bv_{k-1}^\tau-\bm_k\eta_k)^T(\nabla^2q)(\bm_k\bm_k^T\bv_{k-1}^\tau-\bm_k\eta_k)|\cF_{k-1}\right] +\Delta\cM_k^{\hess}.
\end{split}\end{equation}
Rearranging the terms, 
we get
\begin{equation}
\Delta\cM_k^{\hess}=A_kB_k-\EE[A_kB_k|\cF_{k-1}]
\end{equation}
where
\begin{equation}
A_k:=\bm_k^T(\nabla^2q)\bm_k,\quad B_k:=(\bm_k^T\bv_{k-1}^\tau-\eta_k)^2.
\end{equation}
This can be expanded as 
\begin{equation}\label{eq:quad_increment_expansion}\begin{split}
\Delta\cM_k^{\hess}=&(A_k-\EE[A_k])(B_k-\EE[B_k])+\EE[A_k]\EE[B_k]-\EE[A_kB_k]\\
&+(A_k-\EE[A_k])\EE[B_k]+(B_k-\EE[B_k])\EE[A_k],
\end{split}\end{equation}
so we focus first on obtaining subexponential bounds for the quantities $A_k-\EE[A_k]$ and $B_k-\EE[B_k]$ using the Hanson-Wright inequality.  For $A_k$, we have 
\begin{equation}\label{eq:A_HW_bound}\begin{split}
\PP(|A_k-\EE A_k|\geq t)&\leq2\exp\left[-c\min\left(\frac{t^2}{d^{-2+4\e}||\nabla^2q||_{HS}^2},\frac{t}{d^{-1+2\e}||\nabla^2q||}\right)\right]\\
&\leq2\exp[-c'\min(t^2d^{1-4\e},td^{1-2\e})]
\leq2\exp[-c''td^{\frac12-2\e}]
\end{split}\end{equation}
and thus we have the subexponential bound
\begin{equation}\label{eq:A_subexp_bound}
||A_k-\EE[A_k]||_{\psi_1}<Cd^{-\frac12+2\e}.
\end{equation}
Next we obtain a subexponential bound for $B_k$.  For the part of $B_k$ not involving $\eta_k$, we use Hanson-Wright to get 
\begin{equation}\label{eq:hess_bound_HW_Bpart}\begin{split}
&\PP\left(\left|\bm_k^T\bv_{k-1}^\tau(\bv_{k-1}^\tau)^T\bm_k-\EE\bm_k^T\bv_{k-1}^\tau(\bv_{k-1}^\tau)^T\bm_k \right|\geq t\right)\\
&\leq2\exp\left[-c\min\left(\frac{t^2}{d^{-2+4\e}||\bv_{k-1}^\tau(\bv_{k-1}^\tau)^T||_{HS}^2},\frac{t}{d^{-1+2\e}||\bv_{k-1}^\tau(\bv_{k-1}^\tau)^T||}\right)\right]\\
&\leq2\exp[-c\min(t^2d^{2-8\e},td^{1-4\e})].
\end{split}\end{equation}
For the terms involving $\eta_k$, we use the Orlicz bounds from the assumptions in the set-up to obtain
\begin{equation}\label{eq:hess_bound_eta_part}\begin{split}
||\bm_k^T\bv_{k-1}^\tau\eta_k||_{\psi_1}
&\leq||\bm_k^T\bv_{k-1}^\tau||_{\psi_2}\cdot||\eta_k||_{\psi_2}=d^{-\frac12+2\e}d^{-\frac12+\e}\\
&=d^{-1+3\e}.
\end{split}\end{equation}
Since also 
$||\eta_k^2||_{\psi_1}=d^{-1+2\e}$
 combining the bounds \eqref{eq:hess_bound_HW_Bpart} and \eqref{eq:hess_bound_eta_part}, we have
\begin{equation}\label{eq:B_subexp_bound}
||B_k-\EE[B_k]||_{\psi_1}<Cd^{-1+4\e}.
\end{equation}
Furthermore, we have
\begin{equation}\label{eq:AB_expect_bound}
\EE[A_k]=O(1),\qquad\EE[B_k]=O(d^{-1}),
\end{equation}
uniformly for all $k$ based on the assumptions on $\eta_k$ and $\bm_k$.  We now use \eqref{eq:A_subexp_bound}, \eqref{eq:B_subexp_bound}, \eqref{eq:AB_expect_bound} to bound each term of \eqref{eq:quad_increment_expansion} in turn.

To bound the contribution from $(A_k-\EE[A_k])(B_k-\EE[B_k])$, we observe that, for each $k$, with overwhelming probability,
$|A_k-\EE[A_k]|<d^{-\frac12+3\e}$ and  $|B_k-\EE[B_k]|<d^{-1+5\e}$, so we can conclude that, with overwhelming probability,
\begin{equation}\label{eq:quad_mart_bound1}
\sum_{k=1}^n\Big|(A_k-\EE[A_k])(B_k-\EE[B_k])\Big|
<nd^{-\frac32+8\e}
<d^{-\frac12+9\e}.
\end{equation}
For the second term of \eqref{eq:quad_increment_expansion} we have
\begin{equation}
\Big|\EE[A_k]\EE[B_k]-\EE[A_kB_k]\Big|=\Big|\EE\big[(A_k-\EE A_k)(B_k-\EE B_k) \big]\Big|
\leq\EE\Big|(A_k-\EE A_k)(B_k-\EE B_k)\Big|.
\end{equation}
We can bound this quantity using
\begin{equation}\begin{split}
\PP\left(\big|(A_k-\EE A_k)(B_k-\EE B_k)\big|\geq t\right)
&\leq\PP(|A_k-\EE A_k|\geq \sqrt{t})+\PP(|B_k-\EE B_k|\geq \sqrt{t})\\
&\leq 4\exp\left[-c\min(td^{1-4\e},\sqrt{t}d^{1-4\e})\right]
\end{split}\end{equation}
where the bound in the last line comes from combining \eqref{eq:A_HW_bound} and \eqref{eq:B_subexp_bound}.  Using this bound, we obtain
\begin{equation}\begin{split}
\Big|\EE[A_k]\EE[B_k]-\EE[A_kB_k]\Big|&\leq\int_0^\infty x \PP\left(\big|(A_k-\EE A_k)(B_k-\EE B_k)\big|\geq x\right)dx\\
&\leq\int_0^14x\exp(-cxd^{1-4\e})dx+\int_1^\infty4x\exp(-c\sqrt{x}d^{1-4\e})dx
\end{split}\end{equation}
Making the change of variables $y=xd^{1-4\e}$ in the first integral and $z=\sqrt{x}d^{1-4\e}$ in the second integral, this becomes
\begin{equation}
4d^{-2+8\e}\int_0^{d^{1-4\e}}y\exp(-cy)dy+4d^{-4+8\e}\int_{d^{1-4\e}}^\infty z^2\exp(-cz)dz=O(d^{-2+8\e}).
\end{equation}
Thus,
\begin{equation}\label{eq:quad_mart_bound2}
\sum_{k=1}^n\Big|\EE[A_k]\EE[B_k]-\EE[A_kB_k]\Big|=O(nd^{-2+8\e}).
\end{equation}
Finally, we note that the remaining terms of \eqref{eq:quad_increment_expansion}, namely $(A_k-\EE[A_k])\EE[B_k]$ and $(B_k-\EE[B_k])\EE[A_k]$, are martingale increments with 
\begin{equation}
||(A_k-\EE[A_k])\EE[B_k]||_{\psi_1}\leq Cd^{-\frac32+2\e},\qquad 
||(B_k-\EE[B_k])\EE[A_k]||_{\psi_1}\leq Cd^{-1+4\e}.
\end{equation}
Applying the Martingale Bernstein inequality, we conclude
\begin{equation}\begin{split}
\PP&\left(\sup_{1\leq k\leq n}\left|\sum_{j=1}^k(A_j-\EE[A_j])\EE[B_j]+(B_j-\EE[B_j])\EE[A_j]\right|\geq t\right)\\
&\leq2\exp\left(-\min\left\{\frac{t}{c\max\sigma_{k,1}},\frac{t^2}{c\sum_{k=1}^n\sigma_{k,1}}
\right\}\right)\\
&\leq2\exp\left(-\min\left\{Ctd^{1-4\e},Ct^2d^{2-8\e}n^{-1}
\right\}\right).
\end{split}\end{equation}
Thus, for $n\leq d\log d$, we get
\begin{equation}\label{eq:quad_mart_bound3}
\sup_{1\leq k\leq n}\left|\sum_{j=1}^k(A_j-\EE[A_j])\EE[B_j]+(B_j-\EE[B_j])\EE[A_j]\right|\leq d^{-\frac12+5\e}
\end{equation}
with overwhelming probability.  Finally, combining the bounds from \eqref{eq:quad_mart_bound1}, \eqref{eq:quad_mart_bound2}, \eqref{eq:quad_mart_bound3}, we conclude that, for $n\leq d \log d$,
\begin{equation}
\sup_{1\leq k\leq n}|\cM_k^{\hess,\tau}|\leq d^{-\frac12+8\e}
\end{equation}
with overwhelming probability.   This completes the proof of part (\ref{item:Mquad_bound}) of the lemma.  

For part (\ref{item:Equad_bound}), we observe that $\Delta\cE_k^{\hess,\tau}=\EE[A_kB_k]-\EE[A_k]\EE[B_k] + O(d^{-2+4\epsilon})$,
the error terms arising from $\bu_k$ cross terms, so that the bound of $\cE_k^{\hess,\tau}$ follows immediately from \eqref{eq:quad_mart_bound2}.
\end{proof}

\subsection{Martingale for HSGD}
Recall that the HSGD process $\{\bV_t\}$ satisfies the differential equation
\begin{equation}
d\bV_t=-\g\nabla\cR(\bV_t+\btx)dt+\g\sqrt{\tfrac2d\cP(\bV_t+\btx)\bK}dB_t
\end{equation}
where
\begin{equation}
\cP(\bx)=(\bx-\btx)^T\bK^2(\bx-\btx)+\eta^2.
\end{equation}
Using It\^o's Lemma, this gives us
\begin{equation}
q(\bV_t^\tau)=q(\bV_0^\tau)
-\g\int_0^t(\nabla q(\bV_{s}^\tau))^T\nabla\cR(\bV_{s}^\tau+\btx)ds
+\tfrac{\g^2}{2}\int_0^t\cP(\bV_{s}^\tau+\btx)\tr(\tfrac1d\bK (\nabla^2q))ds
+\cM_t^{HSGD,\tau},
\end{equation}
where
\begin{equation}
\cM_t^{HSGD,\tau}=\g\int_0^t(\nabla q(\bV_{s}^\tau))^T\sqrt{\tfrac1d\cP(\bV_s^\tau+\btx)\bK}dB_s.
\end{equation}
\begin{lemma}\label{lem:mart_bound_HSGD}
  For any quadratic $q$ with $\|q\|_{C^2}\leq 1$, after imposing the stopping time $\tau$, the resulting martingale $\cM^{HSGD,\tau}_t$ satisfies
\begin{equation}
  \sup_{0 \leq t \leq n/d} \left[\cM^{HSGD,\tau}_t\right]\leq Cd^{-\frac12+3\e},
\end{equation}
provided $n \leq d\log d$.
\end{lemma}
\begin{proof}
This martingale has quadratic variation
\begin{equation}
\left[\cM^{HSGD,\tau}_t\right]=\frac{\gamma^2}{d}\int_0^t\cP(\bV_s^\tau+\btx)(\nabla q(\bV_s^\tau))^T\nabla^2\cP(\bV_s^\tau+\btx)(\nabla q(\bV_s^\tau))ds
\end{equation}
and, for all $s$, we have the bound
\begin{equation}\begin{split}
&|\cP(\bV_s^\tau+\btx)|\leq||\bV_s^\tau||^2||\bK^2||+\eta^2\leq C_1d^{2\e},\\
\end{split}\end{equation}
Thus, we have
\begin{equation}
  \sup_{0 \leq t \leq n/d} \left[\cM^{HSGD,\tau}_t\right]\leq C(\log d)d^{-1+4\e}
\end{equation}
almost surely.  By the Gaussian tail bound for continuous martingales of bounded quadratic varation,
$\PP(\sup_t[\cM^{HSGD,\tau}_t]>a)<\exp(-a^2/(2C(\log d)d^{-1+4\e}))$.  Taking $t=d^{-\frac12 +3\e}$, we conclude that 
\begin{equation}
  \sup_{0 \leq t \leq n/d} \left[\cM^{HSGD,\tau}_t\right]\leq Cd^{-\frac12+3\e}
\end{equation}
with overwhelming probability.
\end{proof}

\printbibliography
\end{document}